\documentclass[12pt]{article}
\usepackage{amsmath} %
\usepackage[dvips]{color}
\usepackage{subfigure}
\usepackage{graphicx}
\definecolor{dgreen}{rgb}{0,.8,.3}
\definecolor{lblue}{rgb}{.2,.3,.7}

\newtheorem{theorem}{Theorem}

\newtheorem{lemma}{Lemma}

\newtheorem{remark}{Remark}

\numberwithin{equation}{section}
\numberwithin{lemma}{section}
\numberwithin{theorem}{section}

\newcommand{\beq}{\begin{equation}}
\newcommand{\eeq}{\end{equation}}
\newcommand{\beqa}{\begin{eqnarray}}
\newcommand{\eeqa}{\end{eqnarray}}
\newcommand{\beqas}{\begin{eqnarray*}}
\newcommand{\eeqas}{\end{eqnarray*}}
\newcommand{\ba}{\begin{array}}
\newcommand{\ea}{\end{array}}
\newcommand{\bi}{\begin{itemize}}
\newcommand{\ei}{\end{itemize}}
\newcommand{\gap}{\hspace*{2em}}

\def\Arg{{\rm Arg}}
\def\cFr{{\Delta^u_r}}
\def\cK{{\cal K}}
\def\cN{{\cal N}}
\def\cO{{\cal O}}

\def\cX{{\cal X}}

\def\tx{{\tilde x}}

\title{An Efficient Optimization Approach for a Cardinality-Constrained Index Tracking Problem}

\author{
 Fengmin Xu
 \thanks{School of Mathematics and Statistics, Xi'an Jiaotong University, Xi'an, 710049, China
(Email: {\tt fengminxu@mail.xjtu.edu.cn.}~This author was supported by China NSFC projects (No.11101325)
and (No.71371152).}
 \and
 Zhaosong Lu
 \thanks{Department of Mathematics, Simon Fraser University, Burnaby, BC, V5A 1S6, Canada
(Email: {\tt zhaosong@sfu.ca}). This author was supported in part by NSERC Discovery Grant.}
\and
Zongben Xu
\thanks{School of Mathematics and Statistics, Xi'an Jiaotong University, Xi'an, 710049, China
(Email: {\tt zbxu@mail.xjtu.edu.cn.})~This author was supported by National 973 Program of China
(No.2007CB311002) and China NSFC projects (No.70531030).}
}

\date{October 20, 2014 (Revised: March 16, 2015)}

\begin{document}

\maketitle

\begin{abstract}

In the practical business environment, portfolio managers often face business-driven requirements
that limit the number of constituents in their tracking portfolio. A natural index tracking model
is thus to minimize a tracking error measure while enforcing an upper bound on the
number of assets in the portfolio. In this paper we consider such a cardinality-constrained index
tracking model. In particular, we propose an efficient nonmonotone projected gradient (NPG)
method for solving this problem. At each iteration, this method usually solves several projected
gradient subproblems. We show that each subproblem has a closed-form solution, which can be
computed in linear time.  Under some suitable assumptions, we establish that any accumulation
point of the sequence generated by the NPG method  is a local minimizer of the cardinality-constrained
index tracking problem.  We also conduct empirical tests to compare our method with the hybrid
evolutionary algorithm \cite{Evolutionary} and the hybrid half thresholding algorithm \cite{L1/2}
for index tracking. The computational results demonstrate that our approach generally produces
sparse portfolios with smaller out-of-sample tracking error and higher consistency between in-sample
and out-of-sample tracking errors. Moreover, our method outperforms the other two approaches in
terms of speed.

\end{abstract}

%%%%%%%%%%%%%%%%%%%%%%%%%%%%%%%%
\section{Introduction}
Index tracking aims at replicating the performance and risk profile of a given market index, and
constructs a tracking portfolio such that the performance of the portfolio is as close as possible to that
of the market index. Index tracking problem has received a great deal of attention in the literature (see,
for example, \cite{Alexander,Ammann,Beasley2003,Brodie,Korn,Beasley,DeMiguel,Fan,Gilli,lobo,Roll,Rudolf,
Tabata,Evolutionary,GaoLi}). An obvious approach is by full replication of the index. It, however, can cause high
administrative and transaction costs. Also, in the practical business environment, portfolio managers
often face business-driven requirements that limit the number of constituents in their tracking
portfolio. Therefore, index tracking can reduce transaction costs and avoid detaining small and illiquid
assets for the index with a large number of constituents.

In this paper we consider a natural model for index tracking, which minimizes a
quadratic tracking error while enforcing an upper bound on the number of assets in the portfolio.
When short selling is not allowed, this model can be formulated mathematically as
\begin{equation} \label{index-track}
\min\limits_{x\in\cFr} TE(x) := \|y - Rx\|^2/T.
\end{equation}
Here, $x \in \Re^n$ is the weight vector of $n$ index constituents; $y\in\Re^T$ is a sample vector of
portfolio returns over a period of length $T$; $R\in\Re^{T \times n}$ consists of the sample returns of
 index constituents over the same period,
\beq \label{cFr}
\cFr := \left\{x\in\Re^n: \ba{l}
\sum_{i = 1}^n {x_i }  = 1, \ \|x\|_0  \le r \\ [4pt]
0 \le x_i  \le u,  \  i=1,\ldots, n
\ea \right\},
\eeq
$\|x\|_0$ denotes the number of nonzero entries of $x$; and $u \in [1/r,1]$ is an upper bound on
the weight of each index constituent. The sum of error squares is used here
to measure the tracking error between the returns of the index and the returns of a portfolio. We
shall mention that another possible tracking error measure is the weighted sum of error squares. Recently,
Gao and Li \cite{GaoLi} studied a related but different cardinality constrained portfolio selection model,
which minimizes the variance of the portfolio subject to a given expected return and a cardinality
restriction on the assets. They developed some efficient lower bounding schemes and proposed a
branch-and-bound algorithm to solve the model.

Index tracking problem \eqref{index-track} involves a cardinality constraint and is generally
NP-hard. It is thus highly challenging to find a global optimal solution to this problem.
Recently,  Fastrich et al.\ \cite{cardinality} studied a relaxation of \eqref{index-track} by
replacing the cardinality constraint in \eqref{index-track} by imposing an upper bound on the
$l_q$-norm ($0<q<1$) \cite{cxy} of the vector of portfolio weights. Xu et al. \cite{L1/2} considered
a special case of this relaxation with $q=1/2$ and proposed a hybrid half thresholding algorithm for
solving this $l_{1/2}$ regularized index tracking model. Lately, Chen et al.\ \cite{Chen14}
proposed a new relaxation of problem \eqref{index-track}, which minimizes the $l_q$-norm regularized
tracking error. They also proposed an interior point method to solve the model. On the other hand, a
local optimal solution of \eqref{index-track} can be found by the penalty decomposition method and the
iterative hard thresholding method that were proposed in \cite{PD,IHT}, respectively. However, they are
generic methods for a more general class of cardinality-constrained optimization problems. When
applied to problem \eqref{index-track}, these methods may not be efficient since they cannot
exploit the specific structure of the feasible region of problem \eqref{index-track}.

Nonmonotone projected gradient (NPG) methods have widely been studied in the literature, which
incorporate the nonmonotone line search technique proposed in \cite{GrLaLu86} into projected gradient
methods. For example, Birgin et al.\ \cite{E.G.Birgin} studied the convergence of an NPG method for minimizing
a smooth function over a closed convex set. Dai and Fletcher \cite{DaFl05} studied a NPG method for solving
a box-constrained quadratic programming in which Barzilai and Borwein's scheme \cite{BB} is used to choose
initial stepsize.  Recently, Francisco and Baz\'an \cite{ncNPG} proposed an NPG method for minimizing
a smooth objective over a general nonconvex set and showed that it converges a generalized stationary point that
is a fixed point of a certain proximal mapping. It is known that NPG methods generally outperform the classical
(monotone) projected gradient methods in terms of speed and/or solution quality (see, for example, \cite{E.G.Birgin,DaFl05,AuSiTe07,TaZh12}).  In this paper, we
propose a simple NPG method for solving problem \eqref{index-track}. At each iteration, our method usually
solves several projected gradient subproblems. By exploiting the specific structure of the feasible region of
problem \eqref{index-track}, we show that each projected gradient subproblem has a closed-form solution,
which can be computed in {\it linear} time. Moreover, we show that any accumulation point of the sequence
generated by our method is an optimal solution of a related convex optimization problem. Under some suitable
assumption, we further establish that such an accumulation point is a local minimizer of problem
\eqref{index-track}. We also conduct empirical tests to compare our method with the other two approaches
proposed in \cite{Evolutionary,L1/2} for index tracking. The computational results demonstrate that our approach
generally produces sparse portfolios with smaller out-of-sample tracking error and higher consistency between
in-sample and out-of-sample tracking errors. Moreover, our method outperforms the other two approaches in
terms of speed.

The rest of the paper is organized as follows. In section \ref{method} we propose a nonmonotone
projected gradient method for solving a class of optimization problems that include problem
\eqref{index-track} as a special case and establish its convergence. In section \ref{result} we conduct
empirical tests to compare our method with the other two existing approaches for index tracking.  We
present some concluding remarks in section \ref{conclude}.

%%%%%%%%%%%%%%%%%%%%%%%%%%%%%%%%%%%%%%%

\section{Nonmonotone projected gradient method}
\label{method}

In this section we propose a nonmonotone projected gradient (NPG) method for solving the problem
\beq \label{sparse-prob}
\min\limits_{x\in\cFr} f(x),
\eeq
 where $\cFr$ is defined in \eqref{cFr} and $f:\Re^n \to \Re$ is Lipschitz continuously differentiable, that is,  there is a constant $L_f > 0$ such that
\beq \label{lipschitz}
\left\| {\nabla f(x) - \nabla f(y)} \right\| \le L_f\left\| {x - y} \right\| \quad \forall x, y \in \Re^n.
\eeq
Throughout this paper, $\|\cdot\|$ denotes the standard Euclidean norm. It is clear to see that
problem \eqref{sparse-prob} includes \eqref{index-track} as a special case. Therefore, the NPG
method proposed below can be suitably applied to solve problem
\eqref{index-track}.

\gap

\noindent
{\bf Nonmonotone projected gradient (NPG) method for \eqref{sparse-prob}}  \\ [5pt]
Let $0< L_{\min} < L_{\max}$, $\tau>1$, $c>0$, integer $M \ge 0$ be given. Choose an
arbitrary $x^0 \in \cFr$ and set $k=0$.
\begin{itemize}
\item[1)] Choose $L^0_k \in [L_{\min}, L_{\max}]$ arbitrarily. Set $L_k = L^0_k$.
\bi
\item[1a)] Solve the subproblem
\beq \label{subprob}
x^{k+1} \in \Arg\min\limits_{x \in \cFr} \left\{\nabla f(x^k)^T (x-x^k) + \frac{L_k}{2} \|x-x^k\|^2\right\}
\eeq
\item[1b)] If
\beq \label{descent}
f(x^{k+1}) \le \max\limits_{[k-M]_+ \le i \le k} f(x^i) - \frac{c}{2} \|x^{k+1}-x^k\|^2
\eeq
is satisfied, then go to step 2).
\item[1c)] Set $L_k \leftarrow \tau L_k$ and go to step 1a).
\ei
\item[2)]
Set $k \leftarrow k+1$ and go to step 1).
\end{itemize}
%\vspace{-.02in}
\noindent
{\bf end}

\gap

\begin{remark}
\begin{itemize}
\item[(i)]
When $M=0$, the sequence $\{f(x^k)\}$ is monotonically decreasing. Otherwise, it may increase at some  iterations and thus the above method is generally a nonmonotone method.
\item[(ii)] A popular choice of $L^0_k$ is by  the following formula proposed by Barzilai and Borwein \cite{BB}
(see also \cite{E.G.Birgin}):
\[
L^0_k  = \max \left\{ L_{\min } ,\min \left\{ L_{\max } ,\frac{(s^k)^T y^k}{\|s^k\|^2}\right\} \right\},
\]
where $s^k  = x^k  - x^{k - 1}$, $y^k=\nabla f(x^k)-\nabla f(x^{k - 1})$.
\end{itemize}
\end{remark}

\gap

We first show that for each outer iteration of the above NPG method, the number of its inner iterations
is finite.

\begin{theorem} \label{inner-convergence}
For each $k \ge 0$, the inner termination criterion \eqref{descent} is satisfied after at most
\[
\max\left\{\left\lfloor \frac{\log(L_f+c)-\log(L_{\min})}{\log \tau} +1\right\rfloor,1\right\}
\]
 inner iterations.
\end{theorem}

\begin{proof}
Let $\bar L_k$ denote the final value of $L_k$ at the $k$th outer iteration,  and let
$n_k$ denote the number of inner iterations for the $k$th outer iteration. We divide the proof
into two separate cases.

Case 1): $\bar L_k=L^0_k$. It is clear that $n_k=1$.

Case 2): $\bar L_k<L^0_k$. Let $H(x)$ denote the objective function of \eqref{subprob}.
By the definition of $x^{k+1}$,
we know that $H(x^{k+1}) \le H(x^k)$, which implies that
\[
\nabla f(x^k)^T(x^{k+1}-x^k) +
\frac{L_k}{2}\|x^{k+1}-x^k\|^2 \le 0.
\]
In addition, it follows from \eqref{lipschitz} that
\[
f(x^{k+1}) \ \le \ f(x^k)+\nabla f(x^k)^T(x^{k+1}-x^k) +
\frac{L_f}{2}\|x^{k+1}-x^k\|^2.
\]
Combining these two inequalities, we obtain that
\[
f(x^{k+1})  \le   f(x^k) - \frac{L_k-L_f}{2}\|x^{k+1}-x^k\|^2.
\]
Hence, \eqref{descent} holds whenever $L_k \ge L_f+c$. This together with the definition of $\bar L_k$
implies that $\bar L_k/\tau < L_f+c$, that is, $\bar L_k <\tau(L_f+c)$.  In view of the definition of $n_k$,
we further have
\[
L_{\min} \tau^{n_k-1} \le L^0_k \tau^{n_k-1} =  \bar L_k <  \tau(L_f+c).
\]
Hence, $n_k \le \left\lfloor \frac{\log(L_f+c)-\log(L_{\min})}{\log \tau} +1\right\rfloor$.

Combining the above two cases, we see that the conclusion holds.
\end{proof}

\gap

We next establish convergence of the outer iterations of the NPG method.

\begin{theorem} \label{converge}
Let $\{x^k\}$ be the sequence generated by the above NPG method. There hold:
\begin{itemize}
\item[(1)] $\{f(x^k)\}$ converges and $\{\|x^k-x^{k-1}\|\} \to 0$.
 \item[(2)] Let $x^*$ be an arbitrary accumulation point of $\{x^k\}$ and $J^*=\{j: x^*_j \neq 0\}$. Then $x^*$ is a stationary point of the problem
\begin{equation} \label{stat-pt}
\begin{array}{ll}
\min\limits_x & f(x) \\
\mbox{s.t.} & \sum_{i=1}^n x_i = 1, \ 0 \le x_j \le u, \ j  \in  J^*; \\
                  & x_j = 0, \ j \notin J^*.
\end{array}
\end{equation}
Suppose further that $f$ is convex. Then
\begin{itemize}
\item[(2a)]
 $x^*$ is a local minimizer of problem \eqref{sparse-prob}
if $\|x^*\|_0=r$;
\item[(2b)]
 $x^*$ is a minimizer of problem \eqref{stat-pt} if $\|x^*\|_0<r$.
\end{itemize}
\end{itemize}
\end{theorem}

\begin{proof}
(1)  Notice that $f$ is continuous in
 $\Delta =\{x\in\Re^n: \sum^n_{i=1} x_i = 1, \ 0 \le x_i \le u \ \forall i\}$.  Since $\{x^k\} \subset \Delta$, it
follows that $\{f(x^k)\}$ is bounded below. Let $\ell(k)$ be an integer such that $[k-M]_+ \le \ell(k) \le k$ and
\[
f(x^{\ell(k)}) = \max\limits_{[k-M]_+ \le i \le k} f(x^i).
\]
It is not hard to observe from \eqref{descent} that $f(x^{\ell(k)})$ is decreasing. Hence, $\lim_{k \to \infty} f(x^{\ell(k)})=\hat f$ for some $\hat f \in \Re$. Using this relation, \eqref{descent}, and a similar
induction argument as used in \cite{WrNoFi09}, one can show that for all $j\ge 1$,
\[
\lim\limits_{k\to\infty} d^{l(k)-j} = 0, \quad\quad  \lim\limits_{k\to\infty} f(x^{l(k)-j})=\hat f,
\]
where $d^k = x^{k+1}-x^k$ for all $k\ge 0$. In view of these equalities, the uniform continuity of $f$
over $\Delta$, and a similar argument in \cite{WrNoFi09}, we can conclude that $\{f(x^k)\}$ converges and
$\{\|x^k-x^{k-1}\|\} \to 0$.

(2) Let $x^*$ be an arbitrary accumulation point of $\{x^k\}$. Then there exists a subsequence
$\cal K$ such that $\{x^k\}_\cK \to x^*$, which together with $\|x^{k}-x^{k-1}\| \to 0$ implies that
$\{x^{k-1}\}_\cK \to x^*$.  By considering a convergent subsequence of $\cK$ if necessary, assume
without loss of generality that there exists some index set $J$ such that  $x^k_j =0$ for every
$j \notin J, k\in \cK$ and $x^k_j > 0$ for all $j \in J, k\in \cK$. Let $\bar L_k$ denote the final value of
$L_k$ at the $k$th outer iteration. From the proof of Theorem \ref{inner-convergence}, we know that
$\bar L_k \in [L_{\min}, \tau(L_f+c)]$. By the definition of $x^k$, one can see that $x^k$ is a minimizer
of the problem
\[
\min\limits_{x \in \cFr} \left\{ \nabla f(x^{k-1})^T (x-x^{k-1}) + \frac{\bar L_{k-1}}{2} \|x-x^{k-1}\|^2\right\}.
\]
Using this fact and the definition of $J$,  one can observe that $x^k$ is also the minimizer of
the problem
\beq \label{subprob1}
\min\limits_{x \in \Omega}  \left\{\nabla f(x^{k-1})^T (x-x^{k-1}) + \frac{\bar L_{k-1}}{2} \|x-x^{k-1}\|^2 \right\},
\eeq
where
\[
\Omega = \left\{x \in \Re^n: \ba {l}
\sum_{i=1}^n x_i = 1, \ 0 \le x_j \le u, \ j \in  J, \\
  x_j = 0, \ j \notin J.
\ea\right\}.
\]
By the first-order optimality conditions of \eqref{subprob1}, we have
\beq \label{1st-cond-k}
-\nabla f(x^{k-1})  - \bar L_{k-1} (x^k-x^{k-1}) \in \cN_\Omega(x^k) \ \ \ \forall k \in \cK,
\eeq
where $\cN_\Omega(x)$ denotes the normal cone of $\Omega$ at $x$. Using
$\bar L_{k-1} \in [L_{\min}, \tau(L_f+c)]$, $\{x^{k-1}\}_\cK \to x^*$, $\|x^k-x^{k-1}\| \to 0$,
outer continuity of $\cN_\Omega(\cdot)$, and taking limits on both sides of \eqref{1st-cond-k} as $k \in \cK \to \infty$,
one can obtain that
\beq \label{stat-cond}
-\nabla f(x^*)  \in \cN_\Omega(x^*).
\eeq
Let $\tilde \Omega$ be the feasible region of problem \eqref{stat-pt}. Clearly, $J^* \subseteq J$ and
hence $\tilde \Omega \subseteq \Omega$, which implies that $\cN_\Omega(x^*) \subseteq
\cN_{\tilde \Omega}(x^*)$. It then follows from \eqref{stat-cond} that $-\nabla f(x^*)  \in
\cN_{\tilde \Omega}(x^*)$.  Hence, $x^*$ is a stationary point of problem \eqref{stat-pt}.

We next prove statements (2a) and (2b) under the assumption that $f$ is convex.

(2a) Suppose that $\|x^*\|_0=r$ and $f$ is convex. We will show that $x^*$ is a local minimizer
of problem \eqref{sparse-prob}.  Let $\epsilon = \min\{x^*_j: j\in J^*\}$,
 \[
\tilde \cO(x^*;\epsilon) = \{x\in\tilde \Omega: \|x-x^*\| < \epsilon\}, \quad\quad  \cO(x^*;\epsilon) = \{x\in\cFr: \|x-x^*\| < \epsilon\},
\]
where $\tilde \Omega$ is defined above. Since $f$ is convex and $x^*$ is a stationary point of
\eqref{stat-pt}, one can conclude that $x^*$ is a minimizer of problem \eqref{stat-pt}, which implies
that $f(x) \ge f(x^*)$ for all $x\in \tilde \cO(x^*;\epsilon)$. In addition, using the definition of
$\epsilon$ and $|J^*|=r$, it is not hard to observe that $\cO(x^*;\epsilon)=\tilde \cO(x^*;\epsilon)$.
It then follows that $f(x) \ge f(x^*)$ for all $x\in \cO(x^*;\epsilon)$, which implies that $x^*$ is
a local minimizer of problem \eqref{sparse-prob}.

(2b) Suppose that $\|x^*\|_0<r$ and $f$ is convex. Recall from above that $x^*$ is a stationary point
of \eqref{stat-pt}. Moreover, notice that problem \eqref{stat-pt} becomes a convex optimization problem when
$f$ is convex. Therefore,  the conclusion of this statement immediately follows.
\end{proof}

\gap

One can observe that problem \eqref{subprob} is equivalent to
\[
x^{k+1} \in \Arg\min\limits_{x \in \cFr} \left\{\left\|x-\left(x^k-\frac{1}{L_k}\nabla f(x^k)\right)\right\|^2\right\},
\]
which is a special case of a more general  problem
\beq \label{proj-subprob}
\min\limits_{x \in \cFr} \|x-a\|^2
\eeq
for some $a\in\Re^n$.
In the remainder of this section we will show that problem \eqref{proj-subprob}  has a
closed-form solution, and moreover, it can be found in linear time. Before proceeding,
we review a technical lemma established in \cite{PD}.

\begin{lemma} \label{lem1}
Let $\cX_i \subseteq \Re$ and $\phi_i: \Re \to \Re$ for $i=1,\ldots,n$ be given. Suppose
that $r$ is a positive integer and $0 \in \cX_i$ for all $i$. Consider the following
$l_0$ minimization problem:
\beq \label{l0-p1}
\min\left\{\phi(x) = \sum^n_{i=1} \phi_i(x_i): \|x\|_0 \le r, \ x \in \cX_1 \times
\cdots \times \cX_n \right\}.
\eeq
Let $\tx^*_i\in \Arg\min\{\phi_i(x_i): x_i \in \cX_i\}$ and $I^* \subseteq \{1,\ldots, n\}$ be
the index set corresponding to the $r$ largest values of $\{v^*_i\}^n_{i=1}$, where
$v^*_i = \phi_i(0)-\phi_i(\tx^*_i)$ for $i=1, \ldots, n$. Then $x^*$ is an optimal solution of
problem \eqref{l0-p1}, where $x^*$ is defined as follows:
\[
x^*_i = \left\{\ba{ll}
\tx^*_i & \mbox{if} \ i \in I^*; \\
0  & \mbox{otherwise},
\ea\right. \quad i=1, \ldots, n.
\]
\end{lemma}

\gap

We are now ready to establish that problem \eqref{proj-subprob} has a  closed-form solution that
can be computed efficiently.

\begin{theorem} \label{closed-form}
Given any $a\in \Re^n$, let $I^* \subseteq \{1,\ldots, n\}$ be the index set corresponding to
the $r$ largest values of $\{a_i\}^n_{i=1}$. Suppose that $\lambda^*\in \Re$ is such that
\beq \label{lambdas}
\sum\limits_{i \in I^*}  \Pi_{[0,u]}(a_i+\lambda^*)=1,
\eeq
where
\[
 \Pi_{[0,u]}(t) = \left\{\ba{ll}
0 & \mbox{if} \ t \le 0;  \\
t & \mbox{if} \  0 < t <u; \\
u & \mbox{if} \ t \ge u
\ea\right. \quad\quad \forall t\in\Re .
\]
Then $x^*$ is an optimal solution of problem \eqref{proj-subprob}, where $x^*$ is defined as
follows:
\[
x^*_i = \left\{\ba{ll}
\Pi_{[0,u]}(a_i+\lambda^*) & \mbox{if} \ i \in I^*; \\
0  & \mbox{otherwise},
\ea\right. \quad i=1, \ldots, n.
\]
\end{theorem}

\begin{proof}
Let $d(x)$ and $d^*$ denote the objective function and the optimal value of \eqref{proj-subprob},
respectively, and $x^*$ be defined above. We can observe that $\|x^*\|_0 \le r$, $\sum^n_{i=1} x^*_i =1$ and $0 \le x^*_j \le u$ for all $j$, which implies that $x^*$ is a feasible solution of
\eqref{proj-subprob}, namely, $x^*\in\cFr$. Hence, $d(x^*) \ge d^*$.  Let
$\psi(t)=t^2-(t-\Pi_{[0,u]}(t))^2$ for every $t\in\Re$. It is not hard to see that $\psi$ is
differentiable, and moreover,
\[
\psi'(t) = 2t- 2(t-\Pi_{[0,u]}(t)) = 2\Pi_{[0,u]}(t) \ge 0.
\]
Hence, $\psi(t)$ is increasing in $(-\infty,\infty)$. Let $\phi_i(x_i)=(x_i-a_i-\lambda^*)^2$,
$\cX_i=[0,u]$, $\tx^*_i=\arg\min\{\phi_i(x_i): x_i \in \cX_i\}$  and $v^*_i=\phi_i(0)-\phi_i(\tx^*_i)$
for all $i$. One can observe that  $\tx^*_i=\Pi_{[0,u]}(a_i+\lambda^*)$
and $v^*_i=\psi(a_i+\lambda^*)$ for all $i$. By the definition of $I^*$ and the
monotonicity of $\psi$, we conclude that $I^*$ is the index set corresponding to the $r$ largest
values of $\{v^*_i\}^n_{i=1}$. In view of Lemma \ref{lem1} and the definitions of $x^*$ and $\tx^*$,
one can see that $x^*$ is an optimal solution to the problem
\[
\underline{d}^*=\min\limits_{0 \le x \le u,\|x\|_0\leq r} \left\{\|x-a\|^2 - 2 \lambda^*(\sum_{i=1}^n x_i-1)\right\},
\]
and hence,
\[
\underline{d}^*=\|x^*-a\|^2 - 2 \lambda^*(\sum_{i=1}^n x_i^*-1)=\|x^*-a\|^2 =d(x^*).
\]
In addition, we can observe that $d^* \ge \underline{d}^*$. It then follows that $d^* \ge d(x^*)$. Recall that $d(x^*) \ge d^*$. Hence, we have $d(x^*) = d^*$. Using this relation and $x^*\in\cFr$, we
conclude that $x^*$ is an optimal solution of  problem \eqref{proj-subprob}.
\end{proof}

\gap

We next show that a $\lambda^*$ satisfying \eqref{lambdas} can be computed in linear time, which
together with Theorem \ref{closed-form} implies that problem \eqref{proj-subprob} can be solved in
linear time as well.

\begin{theorem} \label{thm:lambdas}
For any $a\in\Re^n$ and $u \ge 1/n$, the equation
\beq \label{root-finding}
h(\lambda) := \sum\limits^n_{i=1}  \Pi_{[0,u]}(a_i+\lambda)-1=0.
\eeq
has at least a root $\lambda^*$, and moreover, it can be computed in $O(n)$ time.
\end{theorem}

\begin{proof}
One can observe that $h$ is continuous in $(-\infty,\infty)$, and moreover,  $h(\lambda) = -1$
when $\lambda$ is sufficiently small and $h(\lambda) =nu-1 \ge 0$ when $\lambda$ is
sufficiently large. Hence, \eqref{root-finding} has at least a root $\lambda^*$.

We next show that a root $\lambda^*$ to \eqref{root-finding} can be computed in $O(n)$ time. Indeed,
it is not hard to observe that $h$ is a piecewise linear increasing function in
$(-\infty,\infty)$ with breakpoints $\{-a_1,\ldots,-a_n,-a_1+u,\ldots,-a_n+u\}$.
Suppose that only $k$ of these breakpoints are distinct and they are arranged in strictly increasing
order $\{\lambda_1 < \ldots < \lambda_k\}$. The value of $h$ at each $\lambda_i$ and the slope
of each piece of $h$ can be evaluated iteratively. Indeed, let $\lambda_0=-\infty$. Observe that
$h(\lambda)=-1$ for all $\lambda \le \lambda_1$. Hence, $h$ has slope $s_0=0$ in $(-\infty, \lambda_1]$ and $h(\lambda_1)=-1$. Suppose that $h$ has slope $s_{i-1}$ in $(\lambda_{i-1}, \lambda_i]$, and that  $h(\lambda_i)$ is already computed, and also that
there are $m_i$ number of $\{-a_1,\ldots,-a_n\}$ and $n_i$ number of $\{-a_1+u,\ldots,-a_n+u\}$ equal to $\lambda_i$. Then the slope of $h$ in $(\lambda_{i},\lambda_{i+1}]$ is $s_i=s_{i-1}+m_i-n_i$, which yields
$h(\lambda_{i+1})=h(\lambda_i)+s_i(\lambda_{i+1}-\lambda_i)$ for $i=1, \ldots,k-1$.  Since $h(\lambda_1)=-1$, $h(\lambda_k)=nu-1 \ge 0$ and $h$ is increasing, there exists some
$1 \le j <k$ such that $h(\lambda_j) < 0$ and $h(\lambda_{j+1}) \ge 0$. If $h(\lambda_{j+1})=0$, then
$\lambda^*=\lambda_{j+1}$ is a root to \eqref{root-finding}. Otherwise, $\lambda^* \in (\lambda_j,\lambda_{j+1})$ and $h(\lambda^*)=0$. Using these facts and the relation $h(\lambda)=h(\lambda_j)+s_j(\lambda-\lambda_j)$ for $\lambda \in (\lambda_j,\lambda_{j+1})$,
we can have
\[
\lambda^* = \lambda_j-h(\lambda_j)/s_j.
\]
In addition, one can observe that the arithmetic operation cost of this root-finding procedure is
$O(n)$.
\end{proof}

\section{Numerical results} \label{result}

In this section, we conduct numerical experiments to compare the performance of the  NPG method
proposed in Section \ref{method} with the hybrid evolutionary algorithm \cite{Evolutionary} and the
hybrid half thresholding algorithm \cite{L1/2} for solving index tracking problems. It shall be
mentioned that the NPG method solves the $l_0$ constrained model \eqref{index-track} with $u=0.5$
while the hybrid evolutionary algorithm solves a mixed integer programming model and the hybrid half
thresholding algorithm \cite{L1/2} solves an $l_{1/2}$ regularized index tracking model. These
three methods were coded in Matlab, and all computations were performed on a HP dx7408 PC (Intel
core E4500 CPU, 2.2GHz,1GB RAM) with Matlab 7.9 (R2009b).

The data sets used in our experiments are selected from the standard ones in
OR-library \cite{OR_library} and the CSI 300 index from China Shanghai-Shenzhen stock market.
For the standard data sets,  weekly prices of the stocks from 1992 to 1997 of Hang Seng (Hong Kong),
DAX 100 (Germany), FTSE (Great Britain), Standard and Poor's 100 (USA), the Nikkei index (Japan), the
Standard and Poor's 500 (USA), Russell 2000 (USA) and Russell 3000 (USA) are used. For CSI 300 index,
the daily prices of 300 stocks from 2011 to 2013 in China stock market are considered.
According to the sample scale, we divide the above data sets into two categories:
small data sets including Hang Seng, DAX 100, FTSE , Standard and Poor's 100, the Nikkei index; and
large data sets including CSI 300, Standard and Poor's 500, Russell 2000 and Russell 3000. As in
Torrubiano and Alberto \cite{Evolutionary}, each data set is partitioned into two subsets: a training set
and a testing set. The training set, called in-sample set, consists of the first half of the data and
is used to compute the optimal index tracking portfolio. We also use the in-sample set  and the
formula for $TE$ given in \eqref{index-track} to calculate the tracking error, which is called in-sample
tracking error ($TEI$) of the portfolio.  The testing set, called out-of-sample set, contains the rest of the
data and is used to test the performance of the resulting optimal index tracking portfolio. In particular,
we use the formula for $TE$ in \eqref{index-track} with $(R,y)$ replaced by the out-of-sample set to
calculate the tracking error, which is called out-sample tracking error ($TEO$) of the portfolio.
\textcolor[rgb]{1.00,0.00,0.00}{In addition, we denote the true sparsity of the optimal output generated by each method by $S_{true}$.}

For the NPG method, we set $L_{\min}=10^{-8}$, $L_{\max}=10^8$, $\tau=2$, $c=10^{-4}$, and
$M=3$ for small data sets and $M=5$ for large data sets. For the hybrid half thresholding algorithm,
the lower and upper bounds are chosen to be 0 and 0.5, respectively.  We terminate
these methods when the absolute change of the approximate solutions over two consecutive iterations
is below $10^{-6}$ or the maximum iteration is $10,000$. For the hybrid evolutionary algorithm,
we set the lower bound to 0, the upper bound to 0.5, initial population size to $100$, mutation
probability to $1\%$, cross probability to $50\%$, and maximum iterations to $10,000$. In addition,
we randomly choose a feasible point of problem \eqref{index-track} as a common initial point for these
three methods.

In order to measure the out-of-sample performance and the consistency between in-sample and
 out-of-sample, we introduce the following two criteria.
\bi
\item
 Consistency: The consistency between in-sample and out-of-sample tracking errors of a portfolio
given by a method $A$  is defined as
\begin{equation*}
Cons(A)=|TEI_A-TEO_A|,
\end{equation*}%
where $TEI_A$ and $TEO_A$ are the in-sample and out-of-sample tracking errors of a portfolio
generated by the method $A$.
%In particular, they are computed as follows:
%\[
%TEI_A = \left\|y - Rx^A\right\|^2/T, \quad\quad TEO_A = \left\|y^o - R^o x^A\right\|^2/T,
%\]
%where $x^A$ is the portfolio generated by the method $A$ when applied to model
%\eqref{index-track} with the sample data $(R,y)$, and $(R^o, y^o) \in\Re^{T \times n} \times \Re^T$
%consists of the returns of index constituents and portfolio over a randomly chosen period of length $T$
%distinct from that of $(R,y)$.
Clearly, the smaller value of $Cons(A)$ means that the portfolio by
$A$ has more consistency between in-sample and out-of-sample tracking errors and thus it is more
robust (or less sensitive) with respect to the sample data used for model \eqref{index-track}.

\item
Superiority of out-of-sample: We define
\begin{equation*}
SupO(A,B)=\frac{TEO_{B}-TEO_{A}}{TEO_{B}}\times 100\%,
\end{equation*}%
where $TEO_{A}$ and $TEO_{B}$ are out-of-sample tracking error of the portfolio by methods $A$
and $B$, respectively. One can see that if $SupO(A,B)>0$, $TEO_{A}$ is smaller than $TEO_{B}$,
i.e., the portfolio by method $A$ is superior to that by method $B$ in terms of  out-of-sample
tracking error; and it is very likely that the portfolio by $A$ has a smaller expected tracking error
and thus it is a better estimation to the underlying statistical regression model.
\ei
%According to size of data sets and different parameter selection of $M$, we conduct three kinds of numerical experiments, that is. the small data sets experiment, the large data sets experiment and parameter $M$ experiment.

\subsection{Results on small data sets}\label{test1}

In this subsection, we compare the performance of the NPG method
with the hybrid evolutionary algorithm \cite{Evolutionary} and the hybrid half thresholding algorithm \cite{L1/2} on five small data sets, which are Hang Seng, DAX 100, FTSE, Standard and
Poor's 100, and Nikkei 225. For convenience of presentation, we abbreviate these three
approaches as $l_0$, MIP and $l_{1/2}$ since they are the methods for $l_0$, MIP and $l_{1/2}$ models,
respectively. In order to compare fairly the performance of these methods, we tailor their model
parameters so that the resulting portfolios have same density (i.e., same number of nonzero
entries).

Numerical results are presented in Tables 1 and 2, where $N$ denotes the number of assets
in a data set.  In particular,  we report in Table 1 in-sample error and out-of sample error of
the portfolios generated by the aforementioned three methods. In Table 2, we report the
 consistency between in-sample and out-of-sample errors, and the superiority of out-of-sample
errors for the portfolios generated by these methods. The number of nonzero portfolios given by
these methods is listed in the column named ``density''. From Table 2, we can make the following
observations.

%\vspace{0.01cm}\\

\bi
\item[(i)]
The $l_0$-based method (i.e., NPG method) generally has higher consistency between in-sample error
and out-of-sample error than the MIP- and $l_{1/2}$-based methods (namely, hybrid evolutionary and
half thresholding algorithms) since $Cons(l_0)<Cons(MIP)$ holds for 100\% (30/30)
instances and $Cons(l_0)<Cons(l_{1/2})$ holds for 77.3\% (22/30) instances.
\item[(ii)]
The $l_0$-based method is generally superior to the MIP- and $l_{1/2}$-based methods in terms of
out-of-sample error since $SupO(l_0,MIP)>0$ holds for 90\% (27/30) instances  and
$SupO(l_0,l_{1/2})>0$ holds for 93.3\% (28/30) instances.
\ei
\begin{table}[ht]\label{biao1}
\caption{\footnotesize{The in-sample and out-of-sample tracking errors on small data sets.}}
{\scriptsize \ \centering \renewcommand\arraystretch{1.2} }
\par
\begin{center}
{\scriptsize
\begin{tabular}{ccccccccccc}
\hline
Index & Density & \multicolumn{3}{c}{$l_0$} &
\multicolumn{3}{c}{MIP} & \multicolumn{3}{c}{$l_{1/2}$} \\
&  & $TEI$ & $TEO$ & $S_{true}$ & $TEI$ & $TEO$ & $S_{true}$ &$TEI$ & $TEO$ & $S_{true}$ \\ \hline
Hang & 5    & 6.23e-5 & 5.17e-5 &  5  & 5.69e-5 & 8.87e-5 &  5  & 8.36e-5 & 7.07e-5  & 5 \\
Seng & 6    & 4.29e-5 & 3.45e-5  & 6  & 4.85e-5 & 7.82e-5 &  6 & 8.58e-5 & 7.19e-5  & 6\\
($N$=31) & 7   &2.37e-5 &3.83e-5  &  7 & 3.26e-5 & 5.38e-5 &  7  & 6.45e-5 & 4.59e-5 & 7\\
& 8 & 2.38e-5 & 2.50e-5 & 8  & 2.06e-5 & 3.09e-5  & 8  & 3.20e-5 & 2.95e-5 & 8\\
& 9 & 2.00e-5 & 2.16e-5  &  9 & 1.95e-5 & 2.80e-5  &  9 & 3.96e-5 & 2.44e-5 & 9\\
& 10 & 1.58e-5 & 1.55e-5 & 10 & 1.86e-5 & 2.77e-5  & 10 & 2.33e-5 & 2.34e-5 & 10
\vspace{0.2cm} \\

DAX & 5 & 4.10e-5 & 1.08e-4 &  5  & 2.21e-5 & 1.02e-4 & 5   & 4.88e-5 & 1.18e-4  & 5 \\
($N$=85) & 6 & 3.07e-5 & 1.00e-4  &  6   & 1.82e-5 & 9.43e-5   & 6 & 3.86e-5 & 1.13e-4 & 6 \\
& 7 & 2.56e-5 & 9.68e-5 &  7  & 1.47e-5 & 1.02e-4 & 7  & 2.47e-5 & 1.04e-4 & 7\\
& 8 & 1.68e-5 & 8.71e-5 & 8 & 1.48e-5 & 8.78e-5   &8 & 2.66e-5 & 9.36e-5 & 8\\
& 9 & 1.54e-5 & 8.23e-5 &  9 & 1.05e-5 & 8.63e-5  & 9 & 3.44e-5 & 9.72e-5 & 9\\
& 10 & 1.88e-5 & 8.11e-5 & 10 & 8.21e-6 & 7.76e-5  & 10 & 2.23e-5 & 1.03e-4  & 10
 \vspace{0.2cm}\\

FTSE & 5 & 1.05e-4 & 8.43e-5 & 5  & 6.92e-5 & 9.87e-5 & 5   & 1.22e-4 & 8.80e-5 & 5\\
($N$=89) & 6 & 7.29e-5 & 8.74e-5 &  6  & 5.50e-5 & 9.14e-5 & 6 & 1.04e-4 & 8.78e-5 & 6 \\
& 7 & 6.83e-5 & 8.18e-5  & 7  & 4.15e-5 & 1.02e-4  &  7 & 6.70e-5 & 9.67e-5 & 7 \\
& 8 & 5.81e-5 & 6.00e-5  & 8  & 3.50e-5 & 7.44e-5  & 8  & 6.11e-5 & 7.10e-5 & 8\\
& 9 & 6.51e-5 & 5.67e-5  & 9 & 2.49e-5 & 8.59e-5 &  9 & 7.08e-5 & 5.72e-5 & 9\\
& 10 & 6.70e-5 & 6.94e-5 & 10 & 2.18e-5 & 8.01e-5  & 10  & 5.43e-5 & 7.27e-5 & 10
 \vspace{0.2cm}\\

S\&P & 5 &8.74e-5 & 8.94e-5 & 5  & 4.50e-5 & 1.14e-4 & 5  & 1.02e-4 & 1.14e-4 & 5\\
($N$=98) & 6 & 5.87e-5 & 8.47e-5 & 6 & 3.37e-5 & 1.01e-4  & 6 & 7.93e-5 & 8.88e-5 & 6 \\
& 7 & 3.51e-5 & 7.69e-5 &  7 & 3.36e-5 & 8.93e-5  & 7  & 6.70e-5 & 7.58e-5 & 7 \\
& 8 & 5.50e-5 & 5.75e-5 &  8 & 2.51e-5 & 7.35e-5  & 8 & 6.41e-5 & 6.58e-5  & 8\\
& 9 & 3.71e-5 & 5.09e-5  & 9 & 2.11e-5 & 5.92e-5  & 9 & 5.78e-5 & 6.56e-5  & 9 \\
& 10 & 2.93e-5 & 4.57e-5 & 10 & 1.85e-5 & 5.10e-5  & 10 & 5.22e-5 & 5.07e-5  & 10
\vspace{0.2cm} \\

Nikkei & 5 & 1.34e-4 & 1.32e-4  & 5 & 6.02e-5 & 1.44e-4   & 5 & 1.22e-4 & 1.43e-4& 5\\
($N$=225) & 6 & 9.48e-5 & 9.92e-5  & 6 & 5.13e-5 & 1.20e-4  & 6  & 8.26e-5 & 9.71e-5  & 6\\
& 7 & 7.72e-5 & 9.77e-5  & 7 & 3.93e-5 & 1.11e-4 &   7 & 6.89e-5 & 1.11e-4 &  7\\
& 8 & 9.24e-5 & 8.70e-5  & 8 & 3.12e-5 & 1.18e-4 &  8  & 7.09e-5 & 9.09e-5 &8 \\
& 9 & 4.87e-5 & 7.68e-5  &  9 & 2.78e-5 & 1.18e-4  &  9 & 4.52e-5 & 8.22e-5  & 9\\
& 10 & 6.39e-5 & 6.75e-5  &  10 & 2.36e-5 & 8.25e-5  & 10  & 5.37e-5 & 6.77e-5 & 10\\
\hline
\end{tabular}
}
\end{center}
\end{table}

\begin{table}[ht]\label{biao2}
\caption{\footnotesize{The comparison on small data sets.}}
{\scriptsize \ \centering \renewcommand\arraystretch{1.2} }
\par
\begin{center}
{\scriptsize
\begin{tabular}{ccccccc}
\hline
Index & Density & \multicolumn{1}{c}{$Cons(l_0)$} &
\multicolumn{1}{c}{$Cons(MIP)$} & \multicolumn{1}{c}{$Cons(l_{1/2})$}& $SupO(l_0,MIP)$ &$SupO(l_0,l_{1/2})$ \\ \hline
%& $r$ & $Cons(l_0)$  & $Cons(MIP)$  &$Cons(l_{1/2})$   \\ \hline
Hang & 5   & 1.05e-5    & 3.18e-5    & 1.29e-5    & 41.7 & 26.8  \\
Seng & 6   & 8.37e-6 & 2.97e-5 & 1.39e-5 & 55.9 & 52.1   \\
($N$=31) & 7 & 1.46e-5 & 2.13e-5 & 1.86e-5 & 28.8 & 16.4 \\
& 8 & 1.23e-6 & 1.03e-5 & 2.43e-6 & 19.0 & 15.3   \\
& 9 & 1.66e-6 & 8.50e-6 & 1.52e-5 & 22.9 & 11.4  \\
& 10 & \textbf{3.54e-7} & 9.15e-6 & \textbf{8.50e-8} & 44.3 & 33.9
\vspace{0.2cm} \\

DAX & 5 & 6.72e-5 & 7.97e-5 & 6.94e-5   & \textbf{-6.28} & 8.47  \\
($N$=85) & 6 & 6.95e-5 & 7.61e-5 & 7.49e-5 & \textbf{-6.27} & 11.7  \\
& 7 & 7.12e-5 & 8.69e-5 & 7.96e-5 & 4.72 & 7.26  \\
& 8 & \textbf{7.03e-5} & 7.30e-5 & \textbf{6.70e-5} & 0.79 & 6.96   \\
& 9 & \textbf{6.69e-5} & 7.58e-5 & \textbf{6.28e-5} & 4.68 & 15.3   \\
& 10 & 6.23e-5 & 6.94e-5 & 8.11e-5 & \textbf{-4.52} & 21.6
\vspace{0.2cm} \\

FTSE & 5 & 2.11e-5 & 2.95e-5 & 3.40e-5 & 14.6 & 4.27  \\
($N$=89) & 6 & 1.45e-5 & 3.64e-5 & 1.66e-5 & 4.41 & 0.42   \\
& 7 & 1.35e-5 & 6.05e-5 & 2.98e-5 & 19.8 & 15.4  \\
& 8 & 1.85e-6 & 3.94e-5 & 9.95e-6 & 19.3 & 15.5 \\
& 9 & 8.39e-6 & 6.11e-5 & 1.36e-5 & 34.0 & 0.74  \\
& 10 & 2.46e-6 & 5.83e-5 & 1.85e-5 & 13.3 & 4.52
\vspace{0.2cm}\\

S\&P & 5 & 2.10e-6 & 6.93e-5 & 1.17e-5 & 21.7 & 21.3 \\
($N$=98) & 6 & \textbf{2.60e-5} & 6.70e-5 & \textbf{9.48e-6} & 15.9 & 4.66  \\
& 7 & \textbf{4.18e-5} & 5.57e-5 & \textbf{8.80e-6} & 13.9 & \textbf{-1.40}   \\
& 8 & \textbf{2.58e-6} & 4.83e-5 & \textbf{1.70e-6} & 21.7 & 12.6\\
& 9 & \textbf{1.38e-5} & 3.81e-5 & \textbf{7.81e-6} & 14.0 & 22.4   \\
& 10 & \textbf{1.64e-5} & 3.25e-5 & \textbf{1.49e-6} & 10.4 & 9.96
\vspace{0.2cm} \\

Nikkei & 5 & 2.10e-6 & 8.39e-5 & 2.14e-5 & 8.28 & 7.81 \\
($N$=225) & 6 & 4.38e-6 & 6.83e-5 & 1.46e-5 & 17.0 & \textbf{-2.11} \\
& 7 & 2.05e-5 & 7.16e-5 & 4.19e-5 & 11.9 & 11.8  \\
& 8 & 5.40e-6 & 8.64e-5 & 2.00e-5 & 26.1 & 4.29   \\
& 9 & 2.81e-5 & 8.98e-5 & 3.70e-5 & 34.8 & 6.60  \\
& 10 & 3.60e-6 & 5.89e-5 & 1.39e-5 & 18.1 & 0.23  \\
\hline
\end{tabular}
}
\end{center}
\end{table}

%All these indicate that the NPG method is an efficient approach for index tracking.

\subsection{Results on large data sets}\label{test2}
In this subsection, we compare the performance of the $l_0$-based method (i.e., NPG method) with
the MIP- and $l_{1/2}$-based methods (namely, hybrid evolutionary and half thresholding algorithms)
on four large data sets, which are Standard and Poor's 100,  Russell 2000, Russell 3000 and the
Chinese index CSI 300.  For a fair comparison of the performance of these methods, we tailor
their model parameters so that the resulting portfolios have same density (i.e., same number of
nonzero entries).
%Except for $\tau=2$ in NPG algorithm, the rest parameter settings are the same as in Section 3.1.

\begin{table}[ht]
\caption{\footnotesize{The in-sample and out-of-sample tracking errors on large data sets.}}
\label{biao3}{\scriptsize \ \centering \renewcommand\arraystretch{1.2} }
\par
\begin{center}
{\scriptsize
\begin{tabular}{ccccccccccc}
\hline
Index & Density & \multicolumn{3}{c}{$l_0$} &
\multicolumn{3}{c}{MIP} & \multicolumn{3}{c}{$l_{1/2}$} \\
& & $TEI$ & $TEO$ &$S_{true}$ & $TEI$ & $TEO$ &$S_{true}$ &$TEI$ & $TEO$ &$S_{true}$ \\ \hline
          & 5   &3.34e-5 &2.19e-5  &  5  &1.21e-5  & 2.43e-5  & 5     & 2.39e-5 & 1.99e-5  & 5 \\
CSI 300   & 6   &2.34e-5 &2.11e-5  &  6 &1.17e-5  & 2.37e-5   & 6   & 1.91e-5 & 2.11e-5  &6 \\
($N$=300) & 7   &1.86e-5 &1.98e-5  &  7 &7.84e-6  &2.36e-5    & 7   & 1.51e-5 &2.09e-5   &7 \\
          & 8   &1.67e-5 &1.68e-5  &  8 &7.68e-6  & 2.04e-5   & 8   & 1.42e-5 &1.92e-5   &8 \\
          & 9   &1.67e-5 &1.54e-5  &  9 &7.23e-6  &1.85e-5    & 9  & 1.26e-5 &1.63e-5 & 9\\
          & 10  &1.13e-5 &1.21e-5  &  10 &6.42e-6  &1.51e-5   & 10    & 1.32e-5 &1.33e-5&10\\
          & 20  &6.29e-6 &7.29e-6  &  20 &2.92e-6  &7.65e-6   & 20    & 6.40e-6 &7.64e-6 & 20\\
          & 30  &3.72e-6 &5.14e-6  &  30 &2.07e-6  &5.20e-6   & 30    & 4.15e-6 &5.55e-6 & 30\\
          & 40  &2.39e-6 &4.17e-6  &  40 &1.58e-6  &7.63e-6   & 40    & 3.05e-6 &5.30e-5 & 40\\
          & 50  &2.87e-6 &3.28e-6  &  50 &1.90e-6  &5.00e-6   & 50     & 2.03e-6 &4.53e-6 & ~50
\vspace{0.2cm}\\

         & 80     &2.85e-6  &7.82e-5  & 80   &2.65e-6 &9.98e-5   & 80    & 1.37e-5 &9.85e-5  & 80  \\
     S\&P & 90    &2.43e-6  &7.52e-5  & 90 &3.01e-6 & 1.24e-4    & 90  & 1.08e-5 &9.98e-5  & 90\\
($N$=457) & 100   &2.13e-6  &7.39e-5  & 100 &2.50e-6 & 9.69e-5   &  100  & 9.08e-6 & 1.04e-4 & 100  \\
          & 120   &1.66e-6  &7.59e-5  & 120 & 2.58e-5 & 1.04e-4  & 120   & 6.42e-6 & 9.35e-5  &120 \\
          & 150   &1.52e-6  &7.95e-5  & 150 &5.64e-6  & 1.25e-4  & 150 & 5.18e-6 & 1.07e-4  & 150\\
          & 200   &1.57e-6  &7.94e-5  & 200  & 2.13e-6 & 9.80e-5 & 200    & 2.72e-6 & 9.09e-5 &~200
\vspace{0.2cm}\\

           & 80     &4.02e-6 &2.07e-4 &  80    &3.62e-6 &2.89e-4   & 80  &2.92e-5 &2.34e-4  & 80  \\
Russell 2000 & 90   &3.51e-6 &2.08e-4 &  90  &4.95e-6 &2.76e-4     & 90 &2.76e-5 &2.45e-4 & 90\\
($N$=1318) & 100    &3.18e-6 &1.70e-4 &  100  &2.61e-6 &2.60e-4    & 100  &2.09e-5 &2.13e-4  & 100  \\
          & 120     &2.32e-6 &1.68e-4 &  120  &2.80e-6 &2.49e-4    & 120  &1.71e-5 &2.61e-4  & 120\\
          & 150     &1.99e-6 &1.94e-4 &  150  &1.16e-5 &2.68e-4    & 150  &1.20e-5 &2.66e-4  & 150 \\
          & 200     &9.83e-7 &2.28e-4 &  200   &1.42e-6 &3.31e-4   & 200    &6.89e-6 &3.18e-4 &~200
\vspace{0.2cm}\\

            & 80   &6.24e-6  &1.34e-4  & 80 &3.90e-6 &1.70e-4  & 80 &2.62e-5 & 1.64e-4  &  80 \\
Russell 3000 & 90  &5.49e-6  &1.14e-4  & 90 &3.33e-6 &1.21e-4  & 90 &1.99e-5 & 1.47e-4  & 90\\
($N$=2151) & 100   &4.10e-6  &1.05e-4  & 100 &3.48e-6 &1.05e-4  &100  &1.87e-5 & 1.37e-4  & 100 \\
          & 120    &2.78e-6  &9.82e-5  & 120 &3.01e-6 &1.06e-4  & 120 &1.66e-5 & 1.26e-4  & 120\\
          & 150    &1.63e-6  &1.00e-4  & 150 &2.48e-6 &1.10e-4  & 150  &1.46e-5 & 1.23e-4  & 150\\
          & 200    &1.41e-6  &1.06e-4  & 200 &3.22e-6 &1.09e-4  & 200 &1.03e-5 & 1.57e-4  &200\\
\hline
\end{tabular}
}
\end{center}
\end{table}

\begin{table}[ht]\label{biao4}
\caption{The comparison on large data sets.}
{\scriptsize \ \centering \renewcommand\arraystretch{1.2} }
\par
\begin{center}
{\scriptsize
\begin{tabular}{cccccccccc}
\hline
Index & Density & \multicolumn{3}{c}{Time} & \multicolumn{1}{c}{$Cons(l_0)$} &
\multicolumn{1}{c}{$Cons(MIP)$} & \multicolumn{1}{c}{$Cons(l_{1/2})$} &  $SupO$ &$SupO$ \\
& &$l_0$  & MIP & $l_{1/2}$ & &&& $(l_0,MIP)$ & $(l_0,l_{1/2})$  \\ \hline
          & 5  &0.0114  &26.7  &1.96  &4.05e-6&1.22e-5&5.65e-6   & 18.2   &\textbf{-5.24}    \\
CSI 300   & 6  &0.0113  &36.0  &2.17  &\textbf{2.32e-6}&1.20e-5&\textbf{1.96e-6}   &10.9  &\textbf{-0.08}  \\
($N$=300) & 7  &0.0039  &42.1  &2.31  &1.22e-6&1.58e-5&5.83e-6   &16.3  &5.41    \\
          & 8  &0.0097  &13.5  &2.18  &1.91e-7&1.27e-5&4.94e-6   &17.3  &12.1    \\
          & 9  &0.0078  &17.1  &2.50  &1.35e-6&1.12e-5&3.66e-6   &16.7  &5.44  \\
          & 10 &0.0053  &14.6  &2.71  &\textbf{7.95e-7}&8.67e-6&\textbf{1.28e-7}   &19.7  &8.72    \\
          & 20 &0.0078  &2.84  &4.30  &1.00e-6&4.73e-6&1.23e-6   &4.65  &4.49  \\
          & 30 &0.0060  &1.97  &6.47  &\textbf{1.42e-6}&3.13e-6&\textbf{1.40e-6 }  &1.21  &7.43    \\
          & 40 &0.0064  &2.20  &6.85  &1.78e-7&6.05e-6&2.25e-6   &45.4  &21.4  \\
          & 50 &0.0083  &1.76  &7.65  &4.10e-7&3.11e-6&2.50e-6   &34.5  &27.8
 \vspace{0.2cm}\\

          & 80   &0.0271  &63.6   &8.64  &7.53e-5&9.72e-5&8.48e-5   &21.7  &20.7     \\
     S\&P & 90   &0.0207  &49.0   &10.2  &7.28e-5&1.21e-4&8.90e-5  &39.1  &24.6 \\
($N$=457) & 100  &0.0199  &77.0   &15.3  &7.17e-5&9.44e-5&9.47e-5   &23.7  &28.8   \\
          & 120  &0.0187  &86.9   &13.3  &7.42e-5&1.02e-4&8.71e-5   &27.3  &18.8    \\
          & 150  &0.0184  &58.7   &13.5  &7.80e-5&1.20e-4&1.01e-4   &36.6  &25.4 \\
          & 200  &0.0197  &689.3  &13.7  &7.78e-5&9.58e-5&8.82e-5    &19.0  &12.7
 \vspace{0.2cm}\\

             & 80   &0.153  &577.7   &35.7  &2.03e-4&2.85e-4&2.05e-4   &28.3  &11.6    \\
Russell 2000 & 90   &0.137  &352.6   &27.5  &2.04e-4&2.71e-4&2.17e-4   &24.7  &15.0  \\
($N$=1318)   & 100  &0.148  &657.8   &38.4  &1.67e-4&2.58e-4&1.92e-4   &34.6  &20.1   \\
             & 120  &0.149  &449.1   &47.2  &1.65e-4&2.46e-4&2.44e-4   &32.6  &35.6   \\
             & 150  &0.113  &50.6    &56.5  &1.92e-4&2.56e-4&2.54e-4   &27.6  &27.3   \\
             & 200  &0.095  &1352.7  &46.4  &2.27e-4&3.29e-4&3.11e-4   &30.9  &28.2
 \vspace{0.2cm}\\

             & 80   &0.626  &861.1    &37.1  &1.28e-4&1.66e-4&1.38e-4 &21.0  &18.6     \\
Russell 3000 & 90   &0.267  &1039.5   &47.9  &1.08e-4&1.18e-4&1.27e-4 &6.00  &22.3   \\
($N$=2151)   & 100  &0.269  &913.1    &48.5  &1.01e-4&1.02e-4&1.19e-4 &0.05  &23.5    \\
             & 120  &0.248  &658.7    &88.0  &9.54e-5&1.03e-4&1.09e-4 &7.26  &21.8  \\
             & 150  &0.216  &878.7    &74.9  &9.83e-5&1.08e-4&1.09e-4 &9.34  &18.9    \\
             & 200  &0.342  &1999.9   &97.9  &1.05e-4&1.05e-4&1.47e-4 &2.30  &32.4 \\
\hline
\end{tabular}
}
\end{center}
\end{table}

Numerical results are reported in Tables 3 and 4, where $N$ denotes the number of assets
in a data set.  In particular,  we present in Table 3  in-sample
error and out-of sample error of the portfolios generated by the above three methods. In
Table 4, we present the CPU time of these methods and superiority of out-of-sample errors of
the portfolios given by these methods. The number of nonzero portfolios given by these methods is
listed in the column named ``density''. We can have the following observations from Table 4.

%\vspace{0.01cm}\\

\bi
\item[(i)]
The $l_0$-based method (i.e., NPG method) generally has higher consistency between in-sample error
and out-of-sample error than the MIP- and $l_{1/2}$-based methods (namely, hybrid evolutionary and
half thresholding algorithms) since $Cons(l_0)<Cons(MIP)$ holds for 100\% (28/28)
instances and $Cons(l_0)<Cons(l_{1/2})$ holds for 89.3\% (25/28) instances.
\item[(ii)]
The $l_0$-based method is generally superior to the MIP- and $l_{1/2}$-based methods in terms of
out-of-sample error since $SupO(l_0,MIP)>0$ holds for all instances and $SupO(l_0,l_{1/2})>0$ holds for 92.9\% (26/28) instances.
\item[(iii)]
The $l_0$-based method also generally outperforms the MIP- and $l_{1/2}$-based methods
in terms of speed.
\ei

\section{Concluding remarks} \label{conclude}

In this paper we proposed an index tracking model with budget, no-short selling and a cardinality constraint. Also,  we developed  an efficient nonmonotone projected gradient  (NPG) method for solving
this model. At each iteration, this method usually solves several projected gradient subproblems. We
showed that each subproblem has a closed-form solution, which can be computed in linear time.
Under some suitable assumptions, we showed that any accumulation point of the sequence
generated by the NPG method  is a local minimizer of the cardinality-constrained index tracking problem. We also conducted empirical tests on the data sets from OR-library \cite{OR_library} and
the CSI 300 index from China Shanghai-Shenzhen stock market to compare our method with the
hybrid evolutionary algorithm \cite{Evolutionary} and the hybrid half thresholding algorithm
\cite{L1/2} for index tracking. The computational results demonstrate that our approach
generally produces sparse portfolios with smaller out-of-sample tracking error and higher
consistency between in-sample and out-of-sample tracking errors. Moreover, our method
outperforms the other two approaches in terms of speed.

We shall mention that the proposed NPG method in this paper can be used to solve the subproblems arising
in the penalty method or augmented Lagrangian method when applied to solve more general problem
\[
\ba{ll}
\min\limits_{x \in \cFr} & f(x) \\
\mbox{s.t.} & g(x) \le 0, \ h(x) = 0
\ea
\]
for some $g:\Re^n \to \Re^p$ and $h:\Re^n \to \Re^q$,  where $\cFr$ is given in \eqref{cFr}.

\section*{Acknowledgment}

The authors would like to thank the two anonymous referees for their constructive comments which
substantially improved the presentation of the paper.

%%%%%%%%%%%%%%%%%%%%%%%%%%%%%%%%%%

\end{document}